\documentclass[a4wide]{amsart}
\usepackage{amssymb}
\usepackage{amsthm}
\usepackage{amsmath}
\usepackage{bm,enumerate}
\usepackage[active]{srcltx}
\usepackage{color}

\theoremstyle{plain}
\newtheorem{Thm}{Theorem}[section]
\newtheorem{Lem}[Thm]{Lemma}

\newtheorem{Cor}[Thm]{Corollary}

\theoremstyle{definition}

\newtheorem{Exam}[Thm]{Example}

\newcommand{\TT}{{\mathbb T}}

\newcommand{\NN}{{\mathbb N}}
\newcommand{\ZZ}{{\mathbb Z}}

\newcommand{\DD}{{\mathbb D}}

\newcommand{\HH}{{\mathcal H}}

\def\<{\langle}
\def\>{\rangle}

\advance\textwidth by 2cm

\begin{document}
\date{}

\title{A short note on the Feichtinger Conjecture} 
\author{I. Chalendar, E. Fricain, D. Timotin}

\address{Universit\'e de Lyon;
  Universit\'e Lyon 1, INSA de Lyon, Ecole Centrale de Lyon, CNRS, UMR 5208, Institut Camille Jordan, 43 boulevard du 11 novembre 1918,
 F-69622 Villeurbanne Cedex, France. 
 {\tt chalenda@math.univ-lyon1.fr}.
 },
 \address{Universit\'e de Lyon,
  Universit\'e Lyon 1, INSA de Lyon, Ecole Centrale de Lyon, CNRS, UMR 5208, Institut Camille Jordan, 43 boulevard du 11 novembre 1918,
  F-69622 Villeurbanne Cedex, France.
 {\tt fricain@math.univ-lyon1.fr}} 
 \address{Institute of Mathematics of the 
 Romanian Academy, PO Box 1-764, Bucharest 014700, Romania. {\tt Dan.Timotin@imar.ro}}
 
 \thanks{This work was supported by ANR Project no. 09-BLAN-005801 ``FRAB''}

\begin{abstract}
We discuss the equivalence of the Feichtinger Conjecture with a weaker variant and we show its connection with a conjecture of Agler--McCarthy--Seip concerning complete Nevanlinna--Pick reproducing kernel spaces. Some new examples of sequences of normalized reproducing kernels satisfying the Feichtinger Conjecture are consequently obtained.
\end{abstract}
\maketitle

\section{Introduction}
The Feichtinger Conjecture in harmonic analysis was stated in 2003 and appeared in print for the first time in \cite{casazza-christensen2003}. It became a topic of high interest and of strong activity since it has 
been shown to be equivalent to the celebrated Kadison--Singer Problem \cite{casazza-tremain}. There are many variations of the Feichtinger Conjecture, all equivalent, but we shall be interested in the version involving Bessel sequences and Riesz sequences. First let us recall the basic definitions.  
Let $\HH$ be a Hilbert space and let $(x_n)_{n\geq 1}$ be a sequence in $\HH$. We say that $(x_n)_{n\geq 1}$ is a \emph{Riesz sequence}  if there exist constants $A,B>0$ such that 
\begin{equation}\label{eq:bessel-sequence}
A\sum_{n\geq 1} |a_n|^2\leq \left\| \sum_{n\geq 1} a_n x_n\right\|^2\leq B\sum_{n\geq 1} |a_n|^2,
\end{equation}
for all finitely supported sequence $(a_n)_{n\geq 1}$ of complex numbers. If only the right hand side inequality is satisfied in~\eqref{eq:bessel-sequence}, then we say that $(x_n)_{n\geq 1}$ is a \emph{Bessel sequence}. Note that a Bessel sequence $(x_n)_{n\geq 1}$ is always bounded above; we say that it is \emph{bounded} if it is bounded away from zero, that is there exists a constant $\delta>0$ such that $\|x_n\|\geq\delta$ for every $n\geq 1$.

\vskip 0.3cm
\noindent
{\bf Conjecture 1. Feichtinger Conjecture (FC)}. Every bounded Bessel sequence in a Hilbert space can be partitioned into finitely many Riesz  sequences.

\vskip 0.3cm
It is easily seen that (FC) is equivalent to the next statement.

\vskip 0.3cm
\noindent
{\bf Conjecture 1$'$. Feichtinger Conjecture (FC$'$)}. Every Bessel sequence of unit vectors in a Hilbert space can be partitioned into finitely many Riesz  sequences.

\vskip 0.3cm

 A sequence $(x_n)_{n\geq 1}$ of unit vectors in $\HH$  is called \emph{separated} if there exists a constant $\gamma<1$ such that 
\begin{equation}\label{eq:separated}
|\langle x_n,x_k\rangle|\leq \gamma,
\end{equation}
for any $n,k\geq 1$, $n\not= k$. This definition allows us to  formulate the following slightly weaker conjecture.

\vskip 0.3cm
\noindent
{\bf Conjecture 2. Weaker Feichtinger Conjecture (WFC)}. Every  separated Bessel sequence of unit vectors in a Hilbert space can be partitioned into finitely many Riesz  sequences.
\vskip 0.3 cm

In this short note we show that (WFC) is actually equivalent to (FC$'$) (and thus to (FC)) and we connect these conjectures, in the context of reproducing kernel Hilbert spaces,  with another conjecture of Agler--McCarthy--Seip. This leads to new examples of sequences of reproducing kernels which satisfy (FC).  

We conclude this introduction with some simple (and well-known) remarks. Consider, for a given sequence $(x_n)_{n\geq 1}$ in a Hilbert space $\HH$, the map $J$ that associates to a finitely supported sequence $(a_n)_{n\geq 1}$ of complex numbers the vector $\sum_{n\geq 1} a_n x_n$. It follows immediately from the definitions that $(x_n)_{n\geq 1}$ is a Bessel sequence if and only if $J$ can be extended to a bounded operator from $\ell^2$ to $\HH$, and it is a Riesz sequence if and only if this map is also bounded below. Its adjoint $J^*$ maps $x\in\HH$ into the sequence $(\langle x,x_m\rangle)_{n\geq 1}$, while $J^*J$ is the \emph{Gramian} $\Gamma$ of the sequence, defined by  $\Gamma=(\Gamma_{n,m})_{n,m\geq 1}$, with $\Gamma_{n,m}=\langle x_n,x_m\rangle$. Thus $(x_n)_{n\geq 1}$ is a Bessel sequence if and only if $\Gamma$ is bounded, and it is a Riesz sequence if and only if $\Gamma$ is bounded and bounded below.

Let us note that the separation  can be viewed as a kind of ``baby version'' of the Riesz sequence condition. Indeed, if we consider the Gramian $\left(\begin{smallmatrix}
            1&      \langle x_n,x_m\rangle\\
            \langle x_m,x_n\rangle& 1
            \end{smallmatrix}\right)$ of two unit vectors $x_n, x_m$, its smallest eigenvalue is $1-|\langle x_n, x_m\rangle|$; thus condition~\eqref{eq:separated} can be viewed as saying that all these $2\times 2$ Gramians have a common lower bound, or, equivalently, that all these two term sequences are Riesz pairs with a common lower bound in~\eqref{eq:bessel-sequence}.

Suppose now that we are given a sequence $(x_n)_{n\geq 1}$ of nonzero vectors in a Hilbert space $\HH$ and we are interested by the relevant properties of the normalized sequence $(\hat x_n)_{n\geq 1}$, where we define $\hat x_n=\frac{x_n}{\|x_n\|}$. These are preserved by the action of an invertible operator, as shown by the next lemma.

\begin{Lem}\label{le:action of invertible}
Suppose $(x_n)_{n\geq 1}$ is a sequence in $\HH$ and $A:\HH\to\HH'$ is an invertible operator. Define $x'_n=Ax_n$. Then:
\begin{enumerate}[(a)] 
\item $(\hat x_n)_{n\geq 1}$ is a Bessel sequence if and only if $(\widehat{ x'}\hskip-1mm{}_n)_{n\geq 1}$ is a Bessel sequence.
\item
$(\hat x_n)_{n\geq 1}$ is a Riesz  sequence if and only if $(\widehat{ x'}\hskip-1mm{}_n)_{n\geq 1}$ is a Riesz  sequence.
\item
$(\hat x_n)_{n\geq 1}$ is separated if and only if $(\widehat{ x'}\hskip-1mm{}_n)_{n\geq 1}$ is separated.
\end{enumerate}
\end{Lem}

\begin{proof}
If we denote $J((a_n))=\sum_n a_n \hat x_n$ and $J'((a_n))=\sum_n a_n \widehat{ x'}\hskip-1mm{}_n$, then $J'=AJD$, where
 $D$ is the diagonal operator on $\ell^2$ with entries $\frac{\|x_n\|}{\|Ax_n\|}$; $D$ is invertible since $A$ is. It follows that $J$ is bounded or bounded below if and only if $J'$ is, and thus the Lemma follows from the comments preceding it.
\end{proof}

\section{(WFC) and (FC$'$)}
We first begin with a simple combinatorial lemma.

\begin{Lem}\label{lem:key-lemma}
Let $k$ be a nonnegative integer. Assume that any two distinct integers  are either friends or enemies (and  cannot be both), and  that given an integer $i\in\NN$, it has at most  $k$ enemies.  Then there is partition of $\NN$ into $k+1$ subsets, $I_1,I_2,\dots,I_{k+1}$, such that the elements of each subset  are friends.
\end{Lem}

\begin{proof}
Assume that we have partitioned $\{1,2,\dots,N\}$ into $k+1$ subsets, $I_1,I_2,\dots,I_{k+1}$, such that the elements of each subset  are friends. If $N+1$ has an enemy in each subset, then $N+1$ has at least $k+1$ enemies, a contradiction. Therefore there exists $i_{N+1}\in\{1,\dots,k+1\}$ such that the elements of $I_{i_{N+1}}$ are friends with $N+1$. Then we add $N+1$ to $I_{i_{N+1}}$. Inductively we construct a partition of $\NN$ as required. 
\end{proof}

\begin{Thm}\label{Thm:main}
Let $\HH$ be a Hilbert space and let $(x_n)_{n\geq 1}$ be a Bessel sequence  of unit vectors in $\HH$. Then $(x_n)_{n\geq 1}$ can be partitioned into finitely many separated Bessel sequences. 
\end{Thm}

\begin{proof}
Given two integers $i,j$, we say that $i,j$ are enemies if 
\[
|\langle x_i,x_j\rangle|^2\geq 1/2.
\]
Since $(x_n)_{n\geq 1}$ is a Bessel sequence, the associated operator $J^*$  is bounded, which means that there exists a constant $C>0$ such that 
\[
\sum_{j=1}^{+\infty}|\langle x_j,x\rangle|^2\leq C\|x\|^2.
\] 
Fix $i\in\NN$. Then we get
\[
\frac{1}{2}\sharp\{j:(i,j)\hbox{ \small{are enemies}}\} \leq \sum_{j:(i,j)\hbox{ \footnotesize{are enemies}}}|\langle x_j,x_i\rangle|^2\leq C.
\]
It follows that $i$ has at most $\lfloor 2C \rfloor+1$ enemies, where $\lfloor \cdot \rfloor$ denotes the integer part. The conclusion follows now from Lemma~\ref{lem:key-lemma}.
\end{proof}

We immediately obtain the following.
\begin{Cor}
The Feichtinger Conjecture (FC) and the weaker Feichtinger Conjecture (WFC) are equivalent.
\end{Cor}

It is natural to ask at this point if a separated Bessel sequence is always a Riesz  sequence; according to Theorem~\ref{Thm:main}, this would give a simple solution to the Feichtinger Conjecture. Not unexpectedly, this is false in general; a simple example is the following.

\begin{Exam}
Consider the function $\phi(z)=1+1/2(z+\bar z)$; it is a positive function on $\TT$, and multiplication with $\phi$ in $L^2(\TT)$ is a bounded, non-invertible, positive operator. Its matrix with respect to the canonical basis is $\Gamma=(\Gamma_{n,m})_{n,m\in \ZZ}$, with
\[
 \Gamma_{n,m}=\begin{cases}
               1& \text{if } n-m=0,\\
1/2 & \text{if } |n-m|=1,\\
0 &\text{otherwise.}
              \end{cases}
\]
If one takes $f_n(z)=\phi^{1/2}(z)z^n$, $z\in\TT$, then 
\[
 \langle f_n, f_m \rangle= \langle \phi z^n, z^m \rangle
 =\frac{1}{2\pi}\int_0^{2\pi}\phi(e^{it})e^{i(n-m)t}\,dt
 =\hat\phi(m-n)=\Gamma_{n,m}.
\]
Since $\Gamma$ is bounded, but not bounded below,   $(f_n)_{n\in\ZZ}$ is a separated Bessel sequence that is not a Riesz  sequence. 
\end{Exam}

However, we will see in the next section that for some special classes of  sequences it is indeed true that the Bessel property together with separability imply the Riesz property. These will be sequences of normalized reproducing kernels in certain reproducing kernel Hilbert spaces.

\section{Reproducing kernel Hilbert spaces}
Let $\HH$ be a Hilbert space of functions on some set $\Omega$ such that evaluation at each point of $\Omega$ is a non-zero continuous functional on $\HH$. By Riesz theorem, for each point $\lambda\in\Omega$, there exists $k_\lambda\in\HH$ such that 
\[
f(\lambda)=\langle f,k_\lambda\rangle,\qquad f\in\HH.
\]
The function $k_\lambda$ is called the reproducing kernel of $\HH$. We denote also by $\widehat k_\lambda$ the normalized reproducing kernel, that is $\widehat k_\lambda=k_\lambda/\|k_\lambda\|$. 

In their recent paper~\cite{LP}, Lata and Paulsen have investigated the Feichtinger Conjecture for sequences of normalized reproducing kernels. Although this appears at first sight as a restrictive condition, it is proved in~\cite{LP} that the truth of (FC) for a special class of normalized reproducing kernels (namely, the so-called \emph{de Branges spaces}) would in fact imply its general validity. Note that the Feichtinger Conjecture for certain de Branges spaces has been proved by Baranov and Dyakonov in~\cite{BD}.

Let us specialize to the case of normalized reproducing kernels our main objects of study. 
A sequence $\Lambda=(\lambda_n)_{n\geq 1}$ of $\Omega$ is called an \emph{interpolating sequence for $\HH$} if  the sequence $(\widehat k_{\lambda_n})_{n\geq 1}$ is a Riesz  sequence; it is \emph{separated} if the corresponding sequence of normalized reproducing kernels $(\widehat k_{\lambda_n})_{n\geq 1}$ is separated in the sense of~\eqref{eq:separated}.  As for the Bessel sequence condition, this is related to embedding properties of the space $\HH$. Namely, one says that a measure $\mu$ on $\Omega$ is a Carleson measure for $\HH$ if $\HH\subset L^2(\mu)$. Then one can see that  $(\widehat k_{\lambda_n})_{n\geq 1}$ is a Bessel sequence if and only if the measure $\mu_{\Lambda}=\sum_{n\geq 1}\|k_{\lambda_n}\|^{-2}\delta_{\lambda_n}$ is a Carleson measure for $\HH$ (where $\delta_\lambda$ is the Dirac measure at point $\lambda$).

Suppose we renorm $\HH$ with an equivalent norm, obtaining thus a  space $\HH'$  with reproducing kernel $k'_\lambda$. The above properties for a sequence $\Lambda=(\lambda_n)_{n\geq 1}$ are invariant with respect to this renormalization; indeed, if we denote by $A:\HH\to\HH'$ the formal identity $Af=f$, then
\[
 f(\lambda)=\langle A f, k'_\lambda\rangle_{\HH'} =\langle  f,A^* k'_\lambda\rangle_\HH,
\]
whence it follows that $k_\lambda=A^* k'_\lambda$. Since $A^*$ is invertible, the assertion is a consequence of Lemma~\ref{le:action of invertible}.

A basic example of reproducing kernel space is the Hardy space $H^2$ on the unit disc $\DD$; in this case we have, for $\lambda\in\DD$, $k_\lambda(z)=(1-\overline{\lambda}z)^{-1}$. It has been noticed by Nikolski~\cite{talk} that any
normalized sequence of reproducing kernels of $H^2$ satisfies the Feichtinger Conjecture. This is proved by noting that the Carleson-Shapiro-Shields Theorem implies that if a sequence $\Lambda$ is separated and $\mu_\Lambda$ is a Carleson measure, then $\Lambda$ is interpolating. 
In the abstract framework of Section~2, this means that in $H^2$ any separated Bessel sequence of normalized reproducing kernels is a Riesz sequence.  Theorem~\ref{Thm:main} implies then the validity of (FC) in this case. It turns out that there is a much larger class where this argument is relevant; but we have to introduce first some new terminology.

A reproducing kernel space $\HH$ will be called a \emph{complete Nevanlinna--Pick} space if:

a) $k_\lambda(z)\not=0$ for all $\lambda, z\in\Omega$.

b) There exists $\omega_0\in\Omega$ such that the function
\[
 1-\frac{k_{\omega_0}(z)k_{\lambda}(\omega_0)}{k_{\omega_0}(\omega_0)k_{\lambda}(z)}
\]
is positive definite. 

This class shares many  properties of the classical Hardy space; in particular, its name comes from the fact that an interpolation result similar to the classical Nevanlinnna--Pick theorem is valid.
It contains several spaces of interest; let us only mention the Drury--Arveson space in several variables or the Dirichlet space. For a complete discussion concerning complete Nevanlinna--Pick spaces  we refer to \cite{Agler-McCarthy}. It was conjectured therein (see also Seip \cite{seip}) that in a complete Nevanlinna--Pick space any separated Bessel sequence of normalized reproducing kernels is a Riesz sequence; or, equivalently:

\vskip 0.3cm
\noindent
{\bf Conjecture 3. Agler--McCarthy--Seip Conjecture (ACSC)}. Suppose that $\HH$ is a complete Nevanlinna--Pick space on $\Omega$. If $\Lambda=(\lambda_n)_{n\geq 1}$ is separated and the measure $\mu_\Lambda$ is a Carleson measure for $\HH$, then $\Lambda$ is an interpolating sequence for $\HH$.

\vskip 0.3 cm

According to Theorem~\ref{Thm:main}, we see that (ACSC) implies (FC) for normalized reproducing kernels in  complete Nevanlinna--Pick spaces. As noted above, (ACSC) is valid for the classical Hardy space $H^2$. Although its truth  has not yet been established in full generality, we will discuss in the next section some cases in which it is known to be valid.

\section{Examples}

A general class of complete Nevanlinna--Pick spaces for which (ACSC) is true has been found by Boe~\cite{Boe}. Consider the following condition for a reproducing kernel space $\HH$:

\medskip
\noindent(*) For any sequence $(\lambda_n)_{n\geq 1}$ on $\DD$, the boundedness on $\ell^2$ of the Gram matrix $(\langle \widehat k_{\lambda_n},\widehat k_{\lambda_m}\rangle)_{n,m}$ implies the boundedness  of the matrix $(|\langle \widehat k_{\lambda_n},\widehat k_{\lambda_m}\rangle|)_{n,m}$.

\medskip
In~\cite{Boe} it is proved that (ACSC) is true in complete Nevanlinna--Pick spaces that satisfy (*). 
An even more general condition that implies the truth of (ACSC) is given in~\cite[page 34]{seip}.
 As a consequence, the Feichtinger Conjecture for normalized reproducing kernels is also valid in such spaces.

As an application, consider, for $\alpha\in [0,1]$, the Dirichlet space $\mathcal D_\alpha$, defined as the class of analytic functions $f$ on $\DD$ such that 
\[
\|f\|_{\mathcal D_\alpha}^2=|f(0)|^2+\int_{\DD}|f'(z)|^2(1-|z|^2)^{1-\alpha}dA(z)<+\infty.
\]
For $\alpha=0$, we recover the Hardy space $H^2$ and for $\alpha=1$, we obtain the classical Dirichlet space. After a renormalization with an equivalent norm, these can be shown to be complete Nevanlinna--Pick spaces (see \cite{Agler-McCarthy, seip}). It has been noticed by Boe that, for $0<\alpha\le 1$, they also satisfy condition (*). Although the kernel for the Hardy space does not satisfy (*), we have noticed above that normalized reproducing kernels therein satisfy (FC). We may then
formulate the following corollary.

\begin{Cor}
Let $0\leq \alpha\leq 1$ and let $\widehat k_\lambda^{\alpha} $ be the normalized reproducing kernel of $\mathcal D_\alpha$ at point $\lambda$. Then, for any sequence $(\lambda_n)_{n\geq 1}$ on $\DD$,  the sequence $(\widehat k_{\lambda_n}^\alpha)_{n\geq 1}$ satisfies the Feichtinger Conjecture.
\end{Cor}

One should note that for the case of the Dirichlet space $\mathcal D_1$, (ASCS) had already been proved in~\cite{MS}.

Another class of spaces for which (ACSC) is true is given by certain local Dirichlet spaces $\mathcal D(\mu)$.
Let $\mu$ be a positive finite Borel measure on $\DD$ and let $\mathcal D(\mu)$ be the space of analytic functions $f$ on $\DD$ such that 
\[
\|f\|^2_{\mathcal D(\mu)}:=|f(0)|^2+\int_\DD |f'(z)|^2 P_\mu(z) dA(z)<+\infty,
\]
where $P_\mu$ is the harmonic extension of $\mu$ to $\DD$, i.e.,
\[
P_\mu(z) = \int_\TT \frac{1-|z|^2}{|\zeta-z|^2}\, d\mu(\zeta), \hspace{1cm} z\in \DD.
\]
In particular, if $\mu$ is taken to be the normalized Lebesgue measure on $\TT$, then  $P_\mu(z)=1$, $z\in\DD$, and
therefore $\mathcal D(\mu)$ is the classical Dirichlet space. 

It has been shown by Shimorin~\cite{shimorin} that, again after a renormalization with an equivalent norm,
the spaces $\mathcal D(\mu)$ are complete Nevanlinna--Pick spaces. Further, in the particular case of measures $\mu$ which are  finite sums of point masses, 
G.~Chac\`on proved in \cite{GR2} that $\mathcal D(\mu)$ satisfies the conjecture (ACSC). Therefore we obtain the following result.

\begin{Cor}
Let $\mu$ be a finite sum of point masses and denote by $\widehat k_\lambda^{\mu}$ the normalized reproducing kernel of $\mathcal D(\mu)$ at point $\lambda$. Then, for any sequence $(\lambda_n)_{n\geq 1}$ on $\DD$,  the sequence $(\widehat k_{\lambda_n}^\mu)_{n\geq 1}$ satisfies the Feichtinger Conjecture.
\end{Cor}

Let us note, finally, that all known examples of reproducing kernel spaces in which interpolating sequences $\Lambda$ can be characterized by separation plus the property that $\mu_\Lambda$ is a Carleson measure can be renormed so as to obtain complete Nevanlinna--Pick spaces. It would be interesting to find a different type of reproducing kernel space for which (ASCS) is true.

\end{document}